\documentclass[10pt,twoside]{amsart}
\usepackage{amssymb}
\usepackage{latexsym}
\usepackage{amsfonts}
\usepackage{amsmath}

\DeclareMathOperator{\Aut}{Aut} \DeclareMathOperator{\rep}{rep}
\DeclareMathOperator{\Rep}{Rep}

  \DeclareMathOperator{\Sp}{Sp} \DeclareMathOperator{\PSp}{PSp}

\DeclareMathOperator{\supp}{supp}

\DeclareMathOperator{\fix}{fix}

 \DeclareMathOperator{\GL}{GL}

\DeclareMathOperator{\Sym}{Sym} \DeclareMathOperator{\n}{{\bf{v}}}

\newtheorem{theorem}{Theorem}[section]
\newtheorem{proposition}[theorem]{Proposition}
\newtheorem{lemma}[theorem]{Lemma}

\theoremstyle{definition}

\newtheorem{remark}[theorem]{Remark}

\newcommand{\F}{\mathbb F}

\renewcommand{\leq}{\leqslant}
\renewcommand{\geq}{\geqslant}

\numberwithin{equation}{section}

\begin{document}

\title[Increasing the minimum distance of codes by twisting]
      {Increasing the minimum distance of codes by twisting}
\author{Marzieh Akbari, Neil I. Gillespie and Cheryl E. Praeger}
\address{[Akbari] Faculty of Mathematics,
     K. N. Toosi University of Technology,
     Tehran, IRAN.}
\email{m.akbari@deni.kntu.ac.ir}
\address{[Gillespie] Heilbronn Institute for Mathematical Research, School of Mathematics,
Howard House, The University of Bristol, Bristol, BS8 1SN, United Kingdom.}
\email{neil.gillespie@bristol.ac.uk}
\address{[Praeger] Centre for the Mathematics of Symmetry and Computation
School of Mathematics and Statistics
The University of Western Australia,
35 Stirling Highway, Crawley, Western Australia, 6009.
Also affiliated with King Abdulaziz University, Jeddah, Saudi Arabia.}
\email{cheryl.praeger@uwa.edu.au}     


\begin{abstract}
Twisted permutation codes, introduced recently by the second and
third authors, are frequency permutation arrays. They are similar
to repetition permutation codes, in that they are obtained by a
repetition construction applied to a smaller code. It was previously shown
that the minimum distance of a twisted permutation code is at
least the minimum distance of a corresponding repetition
permutation code, but in some instances can be larger. We
construct two new infinite families of twisted permutation codes with
minimum distances strictly greater than those for the
corresponding repetition permutation codes.
\end{abstract}

\maketitle

\section{Introduction}
\noindent Powerline communication has been proposed as a solution
to ``{\em the last mile problem}'' in the delivery of reliable
telecommunications at the lowest cost \cite{han vinck, state of
the art}. A \emph{constant composition code} of length $m$ over
an alphabet $Q$ of size $q$ has the property that each codeword
has $p_i$ occurrences of the $i^{th}$ letter of the alphabet,
where the $p_i$ are positive integers such that $\sum{p_i}=m$.
Such codes have been proposed as suitable coding schemes to solve
the narrow band and impulse noise problems associated with
powerline communication \cite{dukes des, dukes dis}.

One approach is to choose $p_i=1$ for all $i$, and in this case each
codeword is a permutation on $m$ letters. An \emph{$(m, d)$
permutation array}, usually denoted by $PA(m, d)$, is a set of
permutations on $m$ letters with the property that the Hamming
distance between any two distinct permutations in this set is at
least $d$.
A larger family of constant composition codes, which properly
contains the permutation arrays, consists of those in which
each letter occurs $p$ times, where $p=m/q$. These codes are
called \emph {frequency permutation arrays} (FPA), and have
been studied, for example, in \cite{dukes des, huc
des, huc mullen}. FPA's also play an important role in the study of \emph{neighbour transitive codes}, introduced by 
the second and third authors \cite{ntrans}. In particular they arise naturally in certain classifications of these codes \cite{diagntrans,asntrans}.

Another practical application of FPA's includes the area of multilevel flash
memories. Flash memory is an electronic non-volatile memory that
uses floating-gate cells to store information \cite{capp}. In
flash memories, cells are organized into blocks, where each block
has a large number $(\approx 10^5)$ of cells \cite{capp}. Given a
set of $n$ cells with distinct charge levels, the rank of a cell
indicates the relative position of its own charge level, and so the
ranks of the $n$ cells induce a permutation of $\{1, 2, \ldots ,
n\}$. Schwartz et al. \cite{jiang,tamo} studied error-correcting codes
for such permutations under the infinity norm, motivated by a novel
storage scheme for flash memories called rank modulation (which uses these permutations).
%
%
As for other applications of flash memory: Shieh and Tasi applied FPAs
to provide multilevel flash memory with error correcting
capabilities, and because of their efficient encoding and decoding
algorithms, FPAs can be used in flash memory systems to represent
information and correct errors caused by charge level fluctuation
\cite{shieh}.
%

\emph{Twisted permutation codes} are examples of frequency
permutation arrays with potentially good minimum distance properties (see
\cite{twisted}). These codes are generated by groups and are
generalisations of repeated permutation codes. In \cite{twisted}, the
second and third authors, with Spiga,  proved that the minimum distance of a
twisted permutation code is at least the minimum distance of a
corresponding repetition permutation code for the group.
Moreover, they gave examples for which the minimum distance
is greater than that achieved by the
repetition construction (see \cite[Table 1]{twisted}).

In this paper, we give two new infinite families of twisted permutation
codes with minimum distance strictly greater than the lower
bound given by the usual repetition construction.
We hope that the ideas behind these constructions might suggest
further new constructions with equal or better improvements in minimum distance.
\emph{We would be interested to know how much improvement in minimum distance is
possible:  can this be quantified?}

Let $T$ be an abstract group, let $\mathcal{I}$ be an ordered
$r$-tuple of permutation representations of $T$ on the set $\{1,
\ldots , q\}$ (with repeats allowed), and let $\rho$ be (any) one of these representations. In
Section \ref{construction} we define the twisted permutation code
$C(T,\mathcal{I})$ and the repetition code $\Rep_r(C(T,\rho))$.
Both are frequency permutation arrays of length $rq$ over the
alphabet $\{1, \ldots, q\}$ with each letter occurring $r$ times
in each codeword. Let $\delta_{tw}$ denote the minimum
distance of $C(T,\mathcal{I})$, and let $\delta_{rep}$ be the minimum
of the minimum distance of $\Rep_r(C(T,\rho))$. We prove the
following.
\begin{theorem}\label{main thm}
The twisted permutation codes described in Table~$\ref{T}$ have
minimum distance $\delta_{tw}$ strictly greater than
$\delta_{rep}$.
\end{theorem}

\begin{table}[h]
\centering
\begin{tabular}{c|c|c|c|c|c}
\hline
$T$ & $r$ & $q$ & $\delta_{tw}$ & $\delta_{tw}-\delta_{\rep}$ & Ref. \\
\hline
$\overline{G_k}$ & $p$ & $p^k$ & $p^{k+1}-p$ & $p^2-p$ & Sec. \ref{sec affine} (Prop. 3.7)\\
$\Sp(4, 2^n)$ & $2$ & $2^{3n}+2^{2n}+2^n+1$ &
$2^{3n+1}+2^{2n
}$ & $2^{2n}$ & Sec. \ref{sec symplectic} (Prop. 4.6)\\
\hline
\end{tabular}\\[0.2cm]
\caption{Twisted Permutation Codes with
$\delta_{tw}>
\delta_{\rep}$.}\label{T}
\end{table}

\section{Definitions and Preliminaries}

\subsection{Codes in Hamming Graphs.} For positive integers
$m, q$, each at least $2$, the Hamming graph $\Gamma=H(m, q)$
is the graph whose vertex set $V(\Gamma)$ is the set of
$m$-tuples with entries from an `alphabet' $Q$ of size
$q$, such that two vertices form an edge if and only if
they differ in precisely one entry.
A code $C$ of length $m$ over $Q$ is a subset of $V(\Gamma)$.

The automorphism group of $\Gamma$, which we denote by
$\Aut(\Gamma)$, is the semi-direct product $B\rtimes L$ where
$B\cong S_q^m$ and $L\cong S_m$, see \cite[Theorem
9.2.1]{brouwer}. Its action is described as follows:
for $g=(g_1,\ldots, g_m)\in B$, $\sigma\in L$
and $\alpha=(\alpha_1,\ldots,\alpha_m)\in V(\Gamma)$,
\begin{align*}
\alpha^g=(\alpha_1^{g_1}, \ldots, \alpha_m^{g_m}), \ \ \
\alpha^\sigma=(\alpha_{1^{\sigma^{-1}}}, \ldots,
\alpha_{m^{\sigma^{-1}}}).
\end{align*}
For all pairs of vertices $\alpha,\beta\in V(\Gamma)$, the \emph{Hamming
  distance} between $\alpha$ and $\beta$, denoted by
$d(\alpha,\beta)$, is defined to be the number of entries in
which the two vertices differ. It is the distance between
$\alpha$ and $\beta$ in $\Gamma$, and so we usually refer
simply to distance, rather than Hamming distance.
%
The \emph{minimum distance, $\delta(C)$,} of a code $C$ is the
smallest distance between distinct codewords of $C$, that is,
$$
\delta(C):={\rm min}\{d(\alpha, \beta) | \alpha, \beta\in C, \alpha\neq\beta\}.
$$
If $C$ consists of exactly one codeword, then we set
$\delta(C)=0$.
A code $C$ is called \emph{distance invariant} if, for all positive integers $i$,
the number of codewords at distance $i$ from a codeword $\alpha\in C$
is independent of the choice of $\alpha$.

\subsection{Permutation Groups} Let $\Omega$ be an arbitrary non-empty set.
We denote by $\Sym(\Omega)$ the group of all permutations of
$\Omega$, called the symmetric group on $\Omega$.  A
\emph{permutation group} on $\Omega$ is a subgroup of
$\Sym(\Omega)$. For $t\in\Sym(\Omega)$ and $\alpha\in\Omega$,
we denote by $\alpha^t$ the image of $\alpha$ under $t$.
Suppose that $G$ is a permutation group on $\Omega$
and $t\in G$. We define the \emph{support of $t$} by
$$
\supp(t)=\{\alpha\in\Omega\,:\,\alpha^t\neq\alpha\},
$$
and the set of \emph{fixed points of $t$} by
$$
\fix(t)=\{\alpha\in\Omega\,:\,\alpha^t=\alpha\}.
$$
Then $\Omega=\supp(t)\cup\fix(t)$ for all $t\in G$.
The minimum value $\min\{|\supp(t)|\,:\,1\neq t\in
G\}$ is called the \emph{minimal degree} of $G$.

Let $G_1$ and $G_2$ be two groups acting on the sets $\Omega_1$
and $\Omega_2$, respectively. Then the two actions are said to be
\emph{permutationally isomorphic} if there exists a bijection
$\lambda: \Omega_1 \rightarrow\Omega_2$ and an isomorphism
$\varphi: G_1 \rightarrow G_2$ such that
\begin{equation}
\lambda(\alpha^g)=\lambda(\alpha)^{g^\varphi} \ \ \textnormal{for
all $\alpha\in\Omega_1$ and $g\in G_1$}.
\end{equation}
The pair $(\lambda, \varphi)$ is called a \emph{permutational
isomorphism}.
\subsection{Construction}\label{construction}
Let $Q=\{1,\ldots,q\}$ and $H(q,q)$ be the Hamming graph of
length $q$ over $Q$. Let $T$ be an abstract group, and let
$\rho:T\longrightarrow S_q$ be a (permutation) representation of
$T$ on $Q$ given by $t\mapsto \rho(t)$. For $t\in T$, we identify
the permutation $\rho(t)$ with the vertex in $H(q,q)$ that
represents its passive form, that is,
$\alpha(t,\rho)=(1^{\rho(t)},\ldots,q^{\rho(t)})\in H(q,q)$. We define
\begin{equation}\label{1code}
C(T,\rho):=\{\alpha(t,\rho)\,:\,t\in T\}
\end{equation}
as a code in $H(q,q)$. Then $C(T,\rho)$ is a $(q,d)$-permutation array,
where $d$ is the minimal degree of $\rho(T)$:

\begin{lemma}\cite[Lemma 3.2]{twisted}\label{2.1lem}
For $t\in T$, the distance $d(\alpha(1_T, \rho), \alpha(t,
\rho))=|\supp(\rho(t))|$, and  $C(T, \rho)$ has minimum
distance $\delta(C(T,\rho))$ equal to the minimal degree
$\min\{|\supp(\rho(t))|\,:\,1\neq t\in
T\}$ of $\rho(T)$.
\end{lemma}

Now we consider a general construction. Let
$\mathcal{I}=(\rho_1,\ldots,\rho_r)$ be an ordered list of
representations from $T$ to $S_q$ (with repetitions allowed) and define
$$
\alpha(t,\mathcal{I}):=(\alpha(t,\rho_1),\ldots,\alpha(t,\rho_r))\in
H(rq,q),
$$
an $r$-tuple of codewords of the form given
in (\ref{1code}).  Set
$$
C(T,\mathcal{I}):=\{\alpha(t,\mathcal{I})\,:\,t\in T\}.
$$
Then $C(T,\mathcal{I})$ is a code in $H(rq,q)$, and is called a
\emph{twisted permutation code}. In particular, if $r=1$
then  $C(T,\mathcal{I})$ is the permutation code  $C(T,\rho_1)$ given in
(\ref{1code}), and if $\rho_1=\cdots=\rho_r$ then
$C(T,\mathcal{I})$ is called the \emph{$r$-fold repetition code} for $T\rho_1$, denoted
$\Rep_r(C(T, \rho_1))$.

\begin{proposition}\label{lower bound}\cite{twisted}
With the notation as above, consider the code $C=C(T,\mathcal{I})$.
Then,
\begin{itemize}
\item[(i)] $\Aut(C)$ has a subgroup acting regularly on $C$; in particular,
$C$ is distance invariant;
\item[(ii)] $|C|=|T/K|$, where $K=\cap_{\rho\in\mathcal{I}}\ker\rho$;
\item[(iii)] $\delta(C)=\min_{t\in T^\#}\{\sum_{\rho\in\mathcal{I}}|\supp(\rho(t))|\}
\geq \min_{\rho\in\mathcal{I}}\{\delta(\Rep_r(C(T,\rho))\}$,
where $T^\#=T\backslash \{1\}$.
\end{itemize}
\end{proposition}

The lower bound for $\delta(C(T,\mathcal{I}))$ given in Proposition~\ref{lower bound}~(iii)
is the bound we wish to improve on!

\section{The affine group}\label{sec affine}

In this section we use affine groups to construct a family of twisted permutation
codes with minimum distance greater than the
lower bound of Proposition \ref{lower bound}. Let $k$ be an
integer and $p$ an odd prime number such that $p>k\geq 2$, and let $V=\mathbb{F}_p^k$ be
the  vector space of $k$-dimensional row vectors over the finite field
$\mathbb{F}_p$. Let $A_k$ denote the $k\times k$ matrix over $\F_p$ whose $i$th row is equal to 
\[ r_i=\begin{cases}
{\bf{0}}&\textnormal{if $i=1$}\\
e_{i-1}&\textnormal{if $i\neq 1$}\\
\end{cases}\]
where $e_i$ denotes the $i$th standard basis vector of $V$.  In particular we have that
\begin{equation}\label{eq:ak}
A_k = \left(
\begin{array}{ccc}
{\bf{0}}\\
e_1\\
\vdots\\
e_{k-1}
\end{array}\right)=\left(
\begin{array}{cccc}
e_2^T&\ldots&e_k^T&{\bf{0}}
\end{array}\right),
\end{equation}
where $T$ denotes matrix transpose. Now define $$B_k:=I_k+A_k$$ where $I_k$ is the $k\times k$ identity matrix over $\F_p$.

\begin{lemma}\label{lem:bk}
For $2\leq k<p$ and any positive integer $i$, we have
\begin{equation}\label{e0}
B_k^i=\left(
\begin{array}{ccccc}
1&0&0&\cdots &0\\
i&1&0&\cdots &0\\
\binom{i}{2}&i&1&\cdots &0\\
 \vdots&\vdots&\vdots&\ddots&\vdots\\
 \binom{i}{k-1}&\binom{i}{k-2}&\binom{i}{k-3}&\cdots &1\\
\end{array}
\right),
\end{equation}
where for $1\leq j\leq k-1$, ${i}\choose{j}$ inside the matrix denotes $\underbrace{1+\ldots+1}_{{i}\choose{j}}$ in $\F_p$. In particular, $B_k^p=I_k$
\end{lemma}

\begin{proof}
First note that by multiplying together the two expressions for $A_k$ in \eqref{eq:ak} we find that
\[
A_k^2 = \left(
\begin{array}{ccccc}
e_3^T&\ldots&e_k^T&{\bf{0}}&{\bf{0}}
\end{array}\right).
\]
Then observe that, by induction, for $1\leq s\leq k-1$,
\[
A_k^s=\left(
\begin{array}{cccccc}
e_{s+1}^T&\ldots&e_k^T&{\bf{0}}&\ldots&{\bf{0}}
\end{array}\right)=\left(
\begin{array}{cc}
{\bf 0}& {\bf 0}\\
I_{k-s}&{\bf 0}\\
\end{array}
\right).
\]
Thus $A_k^k={\bf{0}}$, the zero matrix. Now for any positive integer $i$ we evaluate
$$B_k^i=(I_k+A_k)^i=I_k+{{i}\choose{1}}A_k+{{i}\choose{2}}A_k^2+\ldots+{{i}\choose{i-1}}A_k^{i-1}+A_k^i$$
Using the expression for $A_k^s$ above and the fact that $A_k^k={\bf{0}}$, we deduce the expression for $B_k^i$. Since $k<p$, it follows
that $B_k^p=I_k$.
\end{proof}

For a vector $\n=(v_1,
v_2, \ldots, v_k)\in V$, we denote by $\varphi_{\n}$ the translation by $\n$, and
for each $k\geq2$, we define
\[
G_k:=\Big\langle \varphi_{\n}B_k^i:  \ {\n}\in V,  \  i\geq 1\Big\rangle,
\]
and
\begin{equation}\label{eee111}
\overline{G}_k:=\left\langle \left(
\begin{array}{cc}
1&{\n}\\
{\bf 0}& I\\
\end{array}
\right) \left(\begin{array}{cc}
1& {\bf 0}\\
{\bf 0}& B_{k}^i\\
\end{array}
\right):\ {\n}\in V, \  i\geq 1\right\rangle.
\end{equation}
We also set $\Omega(k,1):= I_k$ and, for a positive integer $i\geq 2$, we set $\Omega(k,i):= I+B_k+B_k^2+\cdots
+B_k^{i-1}$. 
\begin{lemma}\label{lem:omk}
With the notation above, we have for any positive integer $i$ and $2\leq k < p$,
\begin{equation}\label{e1}
\Omega(k,i)=\left(
\begin{array}{ccccc}
i&0&0&\cdots &0\\
 \binom{i}{2}&i&0&\cdots &0\\
 \binom{i}{3}&\binom{i}{2}&i&\cdots &0\\
 \vdots&\vdots&\vdots&\ddots&\vdots\\
 \binom{i}{k}&\binom{i}{k-1}&\binom{i}{k-2}&\cdots &i\\
\end{array}
\right),
\end{equation}
where for $1\leq j\leq k$, ${i}\choose{j}$ inside the matrix denotes $\underbrace{1+\ldots+1}_{{i}\choose{j}}$ in $\F_p$. In particular, $\Omega(k,p)={\bf{0}}$.
\end{lemma}

\begin{proof} First we note the following binomial identity (see, for example, \cite[Identity 135]{count}): for $0\leq\ell\leq i-1$,
$$\sum_{j=0}^{i-1}{{j}\choose{\ell}}=\sum_{j=\ell}^{i-1}{{j}\choose{\ell}}={{i}\choose{\ell+1}}.$$
Using this, and Lemma \ref{lem:bk}, it follows that $a_{s,t}$, the $s,t$ entry of $\Omega(k,i)$, is equal to
\[a_{s,t}=\begin{cases}
0&\textnormal{if $t>s$}\\
\sum_{j=0}^{i-1}{{j}\choose{s-t}}={{i}\choose{s-t+1}}&\textnormal{if $t\leq s$,}\\
\end{cases}\]
giving us the required expression. Since $k<p$, we deduce that $\Omega(k,p)={\bf{0}}$, the zero matrix.
\end{proof}

\begin{remark}\label{rem:1} We observe that for any positive integers $i,j$, $\Omega(k,i+j)=\Omega(k,i)+B_k^i\Omega(k,j)$. Since $B_k^p=I_k$ and $\Omega(k,p)={\bf{0}}$,
this implies that for all $i$, $\Omega(k,i)=\Omega(k,i\,{\rm{mod}}\,p )$, where $i\,{\rm{mod}}\,p$ denotes the integer $i_0$ in $\{1,\ldots p\}$ such that $i-i_0$ is divisible by $p$.
\end{remark}

To each  ${\bf w}\in V$ we associate a map $\tau_{\bf w}:
G_k\rightarrow\overline{G}_k$
which, for each $\n\in V$ and each $i\geq 1$, takes
\[
\varphi_{\n}\longmapsto \left(
\begin{array}{cc}
1&{\n}\\
{\bf 0}& I\\
\end{array}
\right)  \ \textnormal{ and } \   B_k^i \longmapsto \left(
\begin{array}{cc}
1& {\bf w}\Omega(k,i)\\
{\bf 0}& B_k^i\\
\end{array}
\right),
\]
and in general $\varphi_{\n}B_k^i \longmapsto \tau_{\bf w}(\varphi_{\n})\tau_{\bf w}(B_k^i)$.

\begin{lemma}\label{iso}
For each ${\bf w}\in V$, the map $\tau_{\bf w}$ defines an isomorphism from  $G_k$ to $\overline{G}_k$.
\end{lemma}

\begin{proof}
First of all, we show that $\tau_{\bf w}$ is a homomorphism: for each $\n_1,\n_2\in V$ and for each $i,j$,
$$
\begin{array}{lll} \tau_{\bf w}(\varphi_{\n_1}B_k^i)\tau_{\bf w}(\varphi_{\n_2}B_k^j)&=&
\left(
\begin{array}{cc}
1&{\n}_1\\
{\bf 0}& I\\
\end{array}
\right) \left(
\begin{array}{cc}
1& {\bf w}\Omega(k,i)\\
{\bf 0}& B_k^i\\
\end{array}
\right) \left(
\begin{array}{cc}
1&{\n}_2\\
{\bf 0}& I\\
\end{array}
\right) \left(
\begin{array}{cc}
1& {\bf w}\Omega(k,j)\\
{\bf 0}& B_k^j\\
\end{array}
\right)
\\[0.5cm]
&=&\left(
\begin{array}{cc}
1& {\bf w}\Omega(k,i+j)+{\n}_1B_k^{i+j}+{\n}_2B_k^{j}\\
{\bf 0}& B_k^{i+j}\\
\end{array}
\right) \ \ \  \ \ \ \  \mbox{(note that $B_k^iB_k^j=B_k^{i+j}$)}
\\[0.5cm]
&=&\left(
\begin{array}{cc}
1&{\n}_1+{\n}_2B_k^{-i}\\
{\bf 0}& I\\
\end{array}
\right) \left(
\begin{array}{cc}
1& {\bf w}\Omega(k,i+j)\\
{\bf 0}& B_k^{i+j}\\
\end{array}
\right)\\[0.5cm]
&=& \tau_{\bf w}(\varphi_{\n_1+\n_2B_k^{-i}}B_k^i B_k^j)=
\tau_{\bf w}(\varphi_{\n_1}B_k^i\cdot\varphi_{\n_2}B_k^j).
\end{array}
$$
Next we prove that ${\rm Ker}(\tau_{\bf w})=1_G$: if
$\varphi_{\n}B_k^i\in {\rm
Ker}(\tau_{\bf w})$, then
\[
\tau_{\bf w}(\varphi_{\n}B_k^i)=\left(
\begin{array}{cc}
1& {\bf w}\Omega(k,i)+{\n}B_k^i\\
{\bf 0}& B_k^i\\
\end{array}
\right)=I,
\]
and hence $B_k^i=I$, which implies by Lemma~\ref{lem:bk}, Lemma~\ref{lem:omk} and Remark~\ref{rem:1}
that $i=p\,\,{\rm mod}\,p$,
$\Omega(k,p)={\bf 0}$ and ${\n}={\bf 0}$. Thus $\varphi_{\n}B_k^i = I$.

Finally, we show that $\tau_{\bf w}$ is onto: for an arbitrary element
\[
g=\left(
\begin{array}{cc}
1& {\n'}B_k^i\\
{\bf 0}& B_{k}^i\\
\end{array}
\right) \in \overline{G}_k,
\]
there is a vector ${\n}\in V$ such that ${\n'}B_k^i={\bf
w}\Omega(k,i)+{\n}B_k^i$, since $B_k^i$ is invertible. Then
$g=\tau_{\bf w}(\varphi_{\n}B_k^i)$.
\end{proof}

Let $\Omega=\{(1, \n) | \n\in V\}$. Then the map $\lambda :
V\rightarrow \Omega$ given by $\n \mapsto (1, \n)$ is a
bijection, and we deduce the following.

\begin{proposition}
The pair $(\lambda, \tau_{\bf 0})$ is a permutational isomorphism from
the action of $G_k$ on $V$ to the action of $\overline{G}_K$ on
$\Omega$.
\end{proposition}
\begin{proof} From the definition of $\lambda$ we
have
\[
\lambda(\n \varphi_{\bf u}B_k^i)=(1, (\n+{\bf u})B_k^i)\ \
\textnormal{for} \ \ {\bf u}, \n\in V, \ i\geq 1, \
k\geq 2.
\]
On the other hand, it follows from the definition of $\tau_{\bf 0
}$ that
\[
\lambda(\n)\tau_{\bf 0}(\varphi_{\bf u}B_k^i)=(1, \n)\left(
\begin{array}{cc}
1& {\bf u}B_k^i\\
{\bf 0}& B_{k}^i\\
\end{array}
\right)=(1, ({\bf u}+\n)B_k^i).
\]
Thus, $\lambda(\n \varphi_{\bf u}B_k^i)=\lambda(\n)\tau_{\bf
0}(\varphi_{\bf u}B_k^i)$, and hence $(\lambda, \tau_{\bf
0})$ is a permutational isomorphism.
\end{proof}

We will need the following information about fixed points
of elements in this action.

\begin{lemma}\label{L4}
Let $\Omega=\{(1, \n) | \n\in V=\mathbb{F}_p^k\}$, let
$\overline{G}_k$ be as in \eqref{eee111}, and let

\begin{equation}\label{defg}
1\neq g=\left(
\begin{array}{cc}
1&{\n}\\
{\bf 0}& I\\
\end{array}
\right) \left(\begin{array}{cc}
1& {\bf 0}\\
{\bf 0}& B_{k}^i\\
\end{array}
\right)=\left(
\begin{array}{cc}
1& {\n}B_{k}^i\\
{\bf 0} & B_{k}^i\\
\end{array}
\right)\in \overline{G}_k,
\end{equation}
with ${\n}B_{k}^i=(v_1, v_2, \ldots, v_k)\in V$ for some $i\geq 1$.
Then $g$ fixes exactly $0$ or $p$
points in $\Omega$. Moreover,

$(1)$ $g$ fixes $p$ points $\Longleftrightarrow$ $i\neq p\,\,{\rm mod}\,p$ and
$v_k=0$; and

$(2)$ $g$ has no fixed points $\Longleftrightarrow$ $i=p\,\,{\rm mod}\,p$, or
$v_k\neq 0$.
\end{lemma}

\begin{proof}
Suppose $g$ has a fixed point $(1, {\bf x})$ where ${\bf x}=(x_1,
x_2, \ldots, x_k)$. Then by Lemma~\ref{lem:bk} and some easy
calculations, we obtain
\[
\begin{array}{lll}
(1,{\bf x}) = (1, x_1, x_2, \ldots, x_k)g & = & (1, x_1, x_2, \ldots, x_k)\left(
\begin{array}{cc}
1& {\n}B_{k}^i\\
{\bf 0}& B_{k}^i\\
\end{array}
\right)\\[0.5cm] &=& (1, x_1, x_2, \ldots, x_k)\left(
\begin{array}{ccccc}
1&v_1&v_2&\cdots &v_k\\
0&1&0&\cdots &0\\
0&i&1&\cdots &0\\
 \vdots&\vdots&\vdots&\ddots&\vdots\\
0&\binom{i}{k-1}&\binom{i}{k-2}&\cdots &1\\
\end{array}
\right)\\[1.5 cm] &=&
(1, v_1+\sum\limits_{t=0}^{k-1}\binom{i}{t}x_{t+1},
v_2+\sum\limits_{t=0}^{k-2}\binom{i}{t}x_{t+2}, \ldots, v_k+x_k).
\end{array}
\]
Comparing the entries in both sides of the equation, we see that

\begin{equation}\label{eq:first}
x_j = v_j + \sum\limits_{t=0}^{k-j}\binom{i}{t} x_{t+j}.  \ \ \
\textnormal{ for $j=1,\ldots,k$.}
\end{equation}
In particular, \eqref{eq:first} with $j=k$ implies that $v_k=0$.
We claim that $i\ne p\,\,{\rm mod}\,p$. Suppose to the contrary that $i=p\,\,{\rm mod}\,p$.
Then \eqref{eq:first} implies that
$x_j=v_j+x_j$ for all $j$, so $\n={\bf 0}$. In particular, $g=
1$, which is a contradiction.
Hence, if $g$ fixes a point then $i\neq p\,\,{\rm mod}\,p$
and $v_k=0$.

Conversely suppose that  $i\neq p\,\,{\rm mod}\,p$
and $v_k=0$. A point $(1,{\bf x})$ is fixed by $g$ if and only if
the equations \eqref{eq:first} hold.
Now $i\,\,{\rm mod}\,p$ is invertible in $\mathbb{F}_p$ (which we identify with $i$ here), so
\eqref{eq:first} imply that
\begin{equation}\label{eq:second}
x_{j+1}=\begin{cases}-\frac{1}{i}\,v_{k-1},\,\,\,\,\,\,\,\,\,\,\,\,\,\,\,\,\,\,\,\,\,\,\,\,\,\,\,\,\,\,\,\,\,\,\,\,\,\,\textnormal{ if $j=k-1$;}\\
\frac{-1}{i}\big(v_j+\sum\limits_{t=2}^{k-j}\binom{i}{t}x_{t+j}\big),\,\,\,\,\textnormal{
if $j=1,\ldots,k-2$.}
\end{cases}
\end{equation}
Thus the entries $x_2, \ldots, x_k$ are
determined by ${\n}B_{k}^i=(v_1, v_2, \ldots, v_k)$, while $x_1$ is unrestricted;
so $g$ fixes exactly $p$ points of $\Omega$, namely
$(1,\  a, \ x_2, \ \ldots, \ x_k )$ for $a\in
\mathbb{F}_p$.
This completes the proof.
\end{proof}

Let ${\bf w}_r:=(0, 0, \ldots, 0, r)\in V$ where $r\in
\mathbb{F}_p$. Note that, if $r\neq 0$, then the
composition of the isomorphisms $\tau_{{\bf w}_r}$ and $\tau_{{\bf
w}_0}^{-1}$, namely $\tau_{{\bf w}_r}\tau_{{\bf w}_0}^{-1}$, is
an automorphism of $\overline{G}_k$.
We need the following information about these automorphisms.

\begin{lemma}\label{fixpoint}
For each $r\in \mathbb{F}_p$, $\tau_{w_r} \tau_{w_0}^{-1} \in
\rm{Aut}(\overline{G}_k)$.
Let $1\ne g\in \overline{G}_k$, as in $\eqref{defg}$.
\begin{enumerate}
\item[(i)] If $i=p\,\,{\rm mod}\,p$ then for each $r\in \mathbb{F}_p$,
$\tau_{{\bf w}_r}\tau_{{\bf w}_0}^{-1}(g)$ has no fixed points in $\Omega$.
\item[(ii)] If $i\ne p\,\,{\rm mod}\,p$ then there exists a unique $r\in \mathbb{F}_p\setminus\{0\}$
such that $\tau_{{\bf w}_r}\tau_{{\bf w}_0}^{-1}(g)$ fixes $p$ points of $\Omega$,
and for each $s\ne r$,
$\tau_{{\bf w}_s}\tau_{{\bf w}_0}^{-1}(g)$ has no fixed points in $\Omega$.
\end{enumerate}
\end{lemma}

\begin{proof}
It follows from Lemma \ref{iso} that $\tau_{w_r}\tau_{w_0}^{-1}$
is an automorphism of $\overline{G}_k$. Direct calculation gives
\[
\tau_{{\bf w}_r}\tau_{{\bf w}_0}^{-1}(g)=\tau_{{\bf
w}_r}(\varphi_{\n}B_k^i)=\left(
\begin{array}{cc}
1& {\bf w}_r\Omega(k,i)+{\n}B_{k}^i\\
{\bf 0}& B_{k}^i\\
\end{array}
\right).
\]
If $i=p\,\,{\rm mod}\,p$, then $\tau_{{\bf w}_r}\tau_{{\bf w}_0}^{-1}(g)=g$,
and hence, by Lemma \ref{L4}, $\tau_{{\bf w}_r}\tau_{{\bf w}_0}^{-1}(g)$
has no fixed points in $\Omega$ for all $r\in \mathbb{F}_p$, proving part (i).
Now suppose $i\neq p\,\,{\rm mod}\,p$. Then
\[
\tau_{{\bf w}_r}\tau_{{\bf w}_0}^{-1}(g)=\left(
\begin{array}{cc}
1& {\bf u}B_{k}^i\\
{\bf 0} & B_{k}^i\\
\end{array}
\right),
\]
where ${\bf u}B_k^i=(u_1, u_2, \ldots, u_k)={\bf
w}_r\Omega(k,i)+{\n}B_{k}^i$. Thus, using Lemma \ref{lem:omk}, we
obtain
$$
(u_1, u_2, u_3, \ldots, u_k)=(v_1+\binom{i}{k}r, v_2+\binom{i}{k-1}r, v_3+\binom{i}{k-2}r, \ldots,
 v_k+\binom{i}{1}r).
$$
By Lemma \ref{L4}, $\tau_{{\bf w}_r}\tau_{{\bf w}_0}^{-1}(g)$
fixes $p$ points if and only if $u_k=v_k+ir=0$. Since, $i\neq p\,\,{\rm mod}\,p$
we can identify $i$ with $i\,\,{\rm mod}\,p  \in\F_p\backslash\{0\}$, and so the previous expression
holds if and only if
$r = -\frac{v_k}{i}$.
If $s\neq -\frac{v_k}{i}$, then $\tau_{{\bf
w}_s}\tau_{{\bf w}_0}^{-1}(g)$ has no fixed points in $\Omega$, proving part (ii).
\end{proof}

Table~\ref{T1} summarises the information about the support sizes of
$\tau_{{\bf w}_r}\tau_{{\bf w}_0}^{-1}(g)$ derived in the proof of
Lemma \ref{fixpoint}.

\begin{table}[h] \centering
\begin{tabular}{l|c|c|c}
\mbox{The size of} $\supp(\tau_{{\bf w}_r}\tau_{{\bf
w}_0}^{-1}(g))$ & $i=p\,\,{\rm mod}\,p$ &$i\neq p\,\,{\rm mod}\,p,\, r=-v_k/i$ &  $i\neq p\,\,{\rm mod}\,p,\, r\ne -v_k/i$ \\
\hline \hline
$|\supp(\tau_{{\bf w}_r}\tau_{{\bf w}_0}^{-1}(g))|$& $p^k$  &$p^k-p$ & $p^k$ \\
\end{tabular} \\[0.2cm]
\caption{$|\supp(\tau_{{\bf w}_r}\tau_{{\bf w}_0}^{-1}(g))|$ for
$g$ as in \eqref{defg}
 with ${\n}B_k^i=(v_1,\ldots , v_k)$.}\label{T1}
\end{table}

We define a family of twisted permutation codes for the group $T=\overline{G}_k$ as follows.
Let $V=\mathbb{F}_p^k$, a $k$-dimensional vector space over
$\mathbb{F}_p$, let $m=p^k$, and choose an ordering
$({\bf v}_1,{\bf v}_2,\ldots,{\bf v}_m)$ for the $m$ vectors of $V$.
Also let $\Omega=\{(1, \n) | \n\in V\}$. For each  $g\in
T$, and $\tau\in\Aut(T)$, define a vertex
in the Hamming graph of length $m$ over $\Omega$ by
\begin{equation}\label{deftauaction}
\alpha(\tau(g))=((1,{\bf v}_1){\tau(g)},\ldots,(1,{\bf v}_m){\tau(g)}),
\end{equation}
that is, the $i$th entry of
$\alpha(\tau(g))$ is the image of $(1,{\bf v}_i)$ under $\tau(g)$.
Let $\mathcal{I}:=(\tau_{{\bf w}_0}\tau_{{\bf w}_0}^{-1}, \tau_{{\bf w}_1}\tau_{{\bf w}_0}^{-1},\ldots,\tau_{{\bf w}_{p-1}}\tau_{{\bf w}_0}^{-1})$ be the sequence of automorphisms of $T$ in Lemma~$\ref{fixpoint}$.
If we identify $T$ with the permutation group it induces on $\Omega$, then each entry of $\mathcal{I}$ is a permutation representation of $T$.
We associate with $g$ the following vertex of the
Hamming graph of length $mp$ over
$\Omega$:
$$
\alpha(g,\mathcal{I}) := (\alpha(\tau_{{\bf w}_0}\tau_{{\bf w}_0}^{-1}(g)),\alpha(\tau_{{\bf w}_1}\tau_{{\bf w}_0}^{-1}(g)),
\ldots,\alpha(\tau_{{\bf w}_{p-1}}\tau_{{\bf w}_0}^{-1}(g)))
$$
that is, $\alpha(g,\mathcal{I})$ is a $p$-tuple of vertices of the form \eqref{deftauaction}.

We define the permutation code
\[
C(T) = \{\alpha(g) | g\in T\},
\]
and note that $C(T)$ is equal to $\{\alpha(\tau(g)) | g\in T\}$ for each $\tau\in\Aut(T)$. We define
the twisted permutation code relative to $\mathcal{I}$ as
$$
C(T, \mathcal{I})=\{\alpha(g,\mathcal{I}) | g\in T\}.
$$

\begin{proposition}\label{T affine}
Let $T=\overline{G}_k$,
$\mathcal{I}$, and $C(T,\mathcal{I})$ be as above.
Then the minimum distance of
$C(T, \mathcal{I})$ is $p^{k+1}-p$, while the lower bound
given by Proposition $\ref{lower bound}$ for this code is
$p^{k+1}-p^2$.
\end{proposition}

\begin{proof}
By Lemma \ref{2.1lem} and Table~\ref{T1}, the minimum distance of
$C(T)$ is $p^k-p$ for each $r$.
Hence the lower bound on the minimum distance of
$C(T,\mathcal{I})$ given by Proposition~\ref{lower bound}(iii),
namely the minimum distance of the
$p$-fold repetition code generated by $C(T)$, is $p^{k+1}-p^2$.
On the other hand, by Lemma \ref{fixpoint} and Table \ref{T1},
for each $g\in T$ as in \eqref{defg},
$$
\sum_{\tau_{{\bf w}_r}\tau_{{\bf w}_0}^{-1}\in \mathcal{I}}
|\supp(\tau_{{\bf w}_r}\tau_{{\bf w}_0}^{-1}(g))| =
\begin{cases} p^{k+1}-p &\mbox{if $i\neq p\,\,{\rm mod}\,p$}\\
              p^{k+1}   &\mbox{if $i=p\,\,{\rm mod}\,p$.}
\end{cases}
$$
Thus by Proposition \ref{lower bound},
$$
\delta(C(T,\mathcal{I}))= p^{k+1}-p
$$
which is strictly greater
than the lower bound in Proposition \ref{lower bound} (iii).
\end{proof}

\section{The symplectic group}\label{sec symplectic}

In this section, we consider the symplectic group $T=\Sp(4, q)$
over a field of order $q=2^n\geq 2$. We exploit the fact that G has
an outer automorphism $\tau$ which does not map transvections to
transvections. We preserve this notation throughout. Set
$\mathcal{I}=\{\iota, \tau\}$ where $\iota$ denotes the identity map. We
show that the twisted permutation code $C(T,\mathcal{I})$ has minimum
distance strictly greater than the lower bound in
Proposition \ref{lower bound}.

Let $V=\mathbb{F}^4$, the space of $4$-dimensional row vectors over a field
$\mathbb{F}$ of order $q=2^n$, and write
$$
{\rm PG}(V)=\{\langle v\rangle \ | \  v\in V \},
$$
the set of $1$-dimensional subspaces of $V$. For $g\in \GL(4,q)$, we denote by
$$
{\rm Fix}_{{\rm PG}(V)}(g)=\{\langle v\rangle \ | \ \langle
v\rangle^g=\langle v\rangle,  \langle v\rangle\in {\rm PG}(V)
\},
$$
the subset of $1$-spaces fixed setwise by
$g$. Now $V$ admits a non-degenerate symplectic form with
automorphism group $T=\Sp(4, q)$. Let $e_1, e_2, f_1, f_2$ be a
symplectic basis for $V$ such that $\{e_1, f_1\}$ and $\{e_2, f_2\}$
are hyperbolic pairs, so that
\[
V=\langle e_1,f_1\rangle\perp \langle e_2,f_2\rangle.
\]
Then $M:=\langle e_1, e_2\rangle$ is a maximal totally isotropic
subspace with ${\rm dim}M=2=({\rm dim}V)/2$.

We begin by stating without proof the following result concerning the isometry groups of non-degenerate symplectic
bilinear forms. Versions of this theorem go back to Witt's work in \cite{Chevalley} and all versions are commonly 
referred to as Witt's Theorem; the proof of the below statement may be found in the book of Artin \cite[p.121]{artin}.

\begin{theorem}\label{Witt} Let $U$ be a subspace of a non-degenerate symplectic space $V$, and let $f:U\rightarrow V$ be a linear isometry.
Then $f$ can be extended to a linear isometry of $V$, that is, there is a linear isometry $h:V\rightarrow V$ such that $f(u)=h(u)$ for all $u\in U$.
\end{theorem}


\begin{lemma}\label{g conjugate} Let $g\in T=\Sp(4,q)$ and suppose that $\left\langle w\,|\,\langle
w\rangle\in{\rm Fix}_{PG (V)}(g)\right\rangle$ has dimension at
least $3$. Then $g$ is conjugate by an element of $T$ to an element whose
matrix, with respect to the ordered basis $(e_1, f_1, e_2,
f_2)$, has the following form
\begin{equation}\label{newg}
\left(
\begin{array}{cccc}
a&0&0&0\\
d&a^{-1}&0&0\\
0&0&b&0\\
0&0&0&b^{-1}\\
\end{array}
\right),
\end{equation}
for some $a,b,d\in\F$.
\end{lemma}

\begin{proof}
By assumption $V$ contains three linearly independent vectors
$v_1, v_2, v_3$ such that $g$ fixes each $\langle v_i\rangle$
setwise, and thus g fixes setwise $W:=\langle v_1, v_2, v_3
\rangle$. Since $W$ is of odd dimension, its radical
$R:=W\cap W^\perp$ (where $W^\perp=\{w\in W\,|\,(w,u)=0\textnormal{ for all $u\in W$}\}$) must be non-zero, and also $R< W$,
because the maximum dimension of the totally isotropic subspaces of
$V$ is equal to 2. Thus $W/R$ is nontrivial and non-degenerate,
and hence $\dim (W/R)=2$ and $\dim R=1$.
Further, since $R\subseteq W^\perp$,
and since $V$ is non-degenerate, we have
$W=R^\perp$. Now there is a linear isometry $R\rightarrow V$
which maps $R$ to $\langle e_1\rangle$, and
by Theorem \ref{Witt} this extends to an element
of $T$ mapping $R$ to $\langle e_1\rangle$. Replacing $g$, if necessary,
by its conjugate under this element, we may assume that
$R=\langle e_1\rangle$ and hence that $W=R^\perp =\langle e_1,e_2, f_2\rangle$.

We claim that the setwise stabiliser ${\rm Stab}_T(W)$
is transitive on the $1$-spaces contained in $W\setminus R$.
To see this, let $w_1$ and $w_2$
be linearly independent vectors in $W\setminus  R$. Then $U_i:=\langle e_1,
w_i\rangle$ is totally isotropic for $i=1, 2$. There is a unique linear map
$f:U_1\rightarrow V$ which sends
\[
f: e_1 \rightarrow  e_1 ,  \ \   w_1 \rightarrow  w_2,
\]
and this map $f$ is a linear isometry, because $U_i$ is totally isotropic for $i=1,2$.
 Therefore Witt's Theorem
implies that there exists $y$ in $T$ which sends
\[
y: e_1 \rightarrow e_1 ,  \ \   w_1 \rightarrow w_2.
\]
Now $y$ fixes $\langle e_1\rangle$, and hence also $\langle
e_1\rangle^{\perp }=W$, so $y\in {\rm Stab}_T(W)$ and the
claim is established.

In what follows, we may assume without loss that $v_1\in
W\setminus R$. By the previous paragraph, there exists an element
$x\in {\rm Stab}_T(W)$ such that $x$ sends $v_1\rightarrow e_2$
and $e_1\rightarrow e_1$. So $g^x=x^{-1}gx$ fixes $\langle
e_2\rangle$ setwise. Hence $g^x$ fixes $\langle e_1\rangle$,
$\langle e_2\rangle$, and $W$.
Since $g^x$ is a conjugate of our original element $g$,
there exists a vector $w\in W\setminus \langle e_1, e_2\rangle$ such that
$\langle w\rangle$ is fixed by $g^x$. Replacing $w$ by a scalar
multiple if necessary, we may assume that $w=ce_1+de_2+f_2$,
for some $c, d\in\mathbb{F}$.
%
Now $W=\langle e_1, e_2, f_2\rangle=\langle e_1, e_2,
w\rangle$ and it is straightforward to check that the map $\phi :W\rightarrow
V$ which sends
\[
\phi : e_1 \rightarrow  e_1 ,  \ \   e_2 \rightarrow  e_2, \ \
f_2 \rightarrow  w,
\]
defines a linear isometry. Hence by Witt's Theorem, there exists an
element say $y\in T$ such that
\[
y: e_1 \rightarrow  e_1 ,  \ \   e_2 \rightarrow  e_2, \ \ f_2
\rightarrow  w.
\]
Then $g^{xy^{-1}}=yg^xy^{-1}$ fixes $\langle e_1\rangle$, $\langle
e_2\rangle$, and $\langle f_2\rangle$. Now we replace $g$ by
its conjugate $g^{xy^{-1}}$. Then $g$ fixes $W$,
$\langle e_1\rangle$, $\langle e_2\rangle$, and $\langle
f_2\rangle$, and so, for some $a, b, c'$ and $d_i$,
\[
g: e_1 \rightarrow  ae_1 ,  \ \   e_2 \rightarrow  be_2, \ \ f_2
\rightarrow  c'f_2, \ \ f_1\rightarrow
d_1e_1+d_2f_1+d_3e_2+d_4f_2.
\]
Since $g$ preserves the form and $(e_2, f_2)=1$, $(e_1, f_1)=1$, $(e_2,
f_1)=0$, and $(f_1, f_2)=0$, we obtain $c'=b^{-1}$, $d_2=a^{-1}$,
$d_4=0$, and $d_3=0$. This shows that the matrix for $g$ with
respect to the ordered basis $(e_1, f_1, e_2, f_2)$ is as in \eqref{newg}.
\end{proof}

\begin{lemma}\label{transvection} Suppose that $q$ is even and that
$g\in \Sp(4, q)$ with $g\ne1$.
If $g$ is a transvection,
then $g$ has exactly $q^2+q+1$ fixed $1$-spaces in $V$; otherwise
it has at most $2q+2$ fixed $1$-spaces in $V$.
\end{lemma}

\begin{proof}
It is not difficult to see that if a non-scalar element $g\in
\GL(4,q)$ fixes $\langle v_1, v_2\rangle$ setwise and $g$ fixes
$\langle v_1\rangle$ and $\langle v_2\rangle$, then the number of
$1$-spaces in $\langle v_1, v_2\rangle$ fixed setwise by $g$ is
$2$ or $q+1$. Thus if all the $1$-spaces fixed setwise by $g$ lie
in a $2$-space, then there is nothing to prove.

So we assume that $g$ fixes setwise three $1$-spaces $\langle
v_1\rangle$, $\langle v_2\rangle$, and $\langle v_3\rangle$
such that $v_1, v_2, v_3$ are linearly independent. Then $g$ fixes
$W := \langle v_1, v_2, v_3\rangle$ setwise, so, by Lemma \ref{g
conjugate}, conjugating $g$ by an element of $\Sp(4,q)$ if
necessary, we may assume that the matrix for $g$ with respect to
the ordered basis $(e_1, f_1, e_2, f_2)$ is as in \eqref{newg}.

Let $S=\{a, a^{-1}, b, b^{-1}\}$. Let
$v=xe_1+yf_1+ze_2+wf_2\in W\setminus\{0\}$,
and suppose that $g$ fixes $\langle v\rangle$
setwise, so $vg=tv$ for some $t\in \mathbb{F}\setminus\{0\}$.
Then
\begin{align*}
ax+yd &=tx,\\
a^{-1}y &=ty,\\
bz &=tz,\\
b^{-1}w &=tw.
\end{align*}
We find ${\rm Fix}_{{\rm PG}(V)}(g)$ and its
size according to the possibilities for $|S|$ and $d$.
Assume first that $d=0$. In this case $|S|\geq 2$ since
$g\ne 1$ and the only scalar matrix in $\Sp(4,q)$ is the identity.

\begin{itemize}
\item[{(1)}] If $|S|=4$ (i.e., $a$, $a^{-1}$, $b$, $b^{-1}$ are all distinct), then
${\rm Fix}_{{\rm PG}(V)}(g)=\{\langle e_1\rangle, \langle
f_1\rangle, \langle e_2\rangle, \langle f_2\rangle\}$, and so
$|{\rm Fix}_{{\rm PG}(V)}(g)|=4< 2q+2$ (and $g$ is not a transvection).

\item[{(2)}] If $|S|=3$, then either $a=a^{-1}$ or $b=b^{-1}$ and we find
\[
{\rm Fix}_{{\rm PG}(V)}(g)=\left\{\begin{array}{lll} \Big\{\langle
xe_1+yf_1\rangle, \langle e_2\rangle, \langle f_2\rangle \ | \ x,
y\in \mathbb{F}, (x, y)\neq (0, 0) \Big\} & \mbox{if}  &
a=a^{-1},\\[0.3cm]
\Big\{\langle e_1\rangle, \langle f_1\rangle, \langle
ze_2+wf_2\rangle | z, w\in \mathbb{F}, (z,w )\ne (0, 0)\Big\} &
\mbox{if} & b=b^{-1}.\end{array}\right.
\]
In both cases, we see
that $|{\rm Fix}_{{\rm PG}(V)}(g)|=q+3\leq 2q+2$ (and $g$ is not a transvection).

\item[{(3)}] If $|S|=2$, then

\[ {\rm Fix}_{{\rm PG}(V)}(g)=\left\{\begin{array}{lll}
\Big\{\langle xe_1+yf_1\rangle, \langle ze_2+wf_2\rangle | x, y,
z, w\in \mathbb{F}, (x,y),(z,w)\ne (0,0 ) \Big\} & \mbox{if}  &
a=a^{-1}, b=b^{-1},\\[0.3cm]
\Big\{\langle xe_1+ze_2\rangle, \langle yf_1+wf_2\rangle | x, y,
z, w\in \mathbb{F}, (x,z),(y,w)\ne (0,0 )\Big\} & \mbox{if} &
a^{-1}\ne a=b,\\[0.3cm]
\Big\{\langle xe_1+wf_2\rangle, \langle yf_1+ze_2\rangle | x, y,
z, w\in \mathbb{F}, (x,w),(y,z)\ne (0,0 )\Big\} & \mbox{if} &
a^{-1}\ne a= b^{-1}.\end{array}\right.
\]
In all cases, we have $|{\rm
Fix}_{{\rm PG}(V)}(g)|=2q+2< 1+q+q^2$ (and $g$ is not a transvection).
\end{itemize}

Now we assume that $d\neq 0$.

\begin{itemize}
\item[{(1)}] If $|S|=4$ (i.e., $a$, $a^{-1}$, $b$, $b^{-1}$ are all distinct), then
\[
{\rm Fix}_{{\rm PG}(V)}(g)=\Big\{\langle e_1\rangle, \langle
\frac{d_1}{(a^{-1}-a)}e_1+f_1\rangle, \langle e_2\rangle, \langle
f_2\rangle\Big\},
\]
and $|{\rm
Fix}_{{\rm PG}(V)}(g)|=4< 2q+2$ (and $g$ is not a transvection).

\item[{(2)}] If $|S|=3$, then either $a=a^{-1}$ or $b=b^{-1}$ and we find
\[
{\rm Fix}_{{\rm PG}(V)}(g)= \left\{\begin{array}{lll}
\{\langle e_1\rangle, \langle e_2\rangle, \langle f_2\rangle\} & \mbox{if} & a=a^{-1}\\
\Big\{\langle e_1\rangle, \langle
\frac{d_1}{(a^{-1}-a)}e_1+f_1\rangle, \langle ze_2+wf_2\rangle |
z, w\in \mathbb{F}, (z, w)\ne (0,0) \Big\}  & \mbox{if} & b=b^{-1}\\
\end{array}\right.
\]
and $|{\rm Fix}_{{\rm PG}(V)}(g)|=3$ or $q+3$ respectively, so is less that
$2q+2$ (and $g$ is not a transvection).

\item[{(3)}]
If $|S|=2$, then
\[{\rm Fix}_{{\rm PG}(V)}(g)=
\Big\{\langle e_1\rangle, \langle ze_2+wf_2\rangle | z, w\in
\mathbb{F}, (z, w)\ne (0, 0)\Big\}, \] if $a=a^{-1}$ and
$b=b^{-1}$, and hence in this case  $|{\rm Fix}_{{\rm
PG}(V)}(g)|=q+2$. Similarly, we obtain \small{
\[ {\rm Fix}_{{\rm
PG}(V)}(g)=\left\{\begin{array}{lll} \Big\{\langle
xe_1+ze_2\rangle, \langle
\frac{d_1y}{(a^{-1}-a)}e_1+yf_1+wf_2\rangle | x,y,z,w\in
\mathbb{F}, (x,z), (y, w)\ne(0,0)\Big\} & \mbox{if} &
a=b,\\[0.3cm]
\Big\{\langle xe_1+wf_2\rangle, \langle
\frac{d_1y}{(a^{-1}-a)}e_1+yf_1+ze_2\rangle| x,y,z,w\in
\mathbb{F}, (x,w), (y, z)\ne(0,0)\Big\} & \mbox{if} & a=
b^{-1}.\end{array}\right.
\] }
and in these two cases $|{\rm Fix}_{{\rm PG}(V)}(g)|=2q+2$
 (and $g$ is not a transvection).

\item[{(4)}] If $|S|=1$, then $a=b=a^{-1}$ and $g$ has determinant
$a^4=1$, so $a=1$ (since $q$ is even). Then
\[
{\rm Fix}_{{\rm PG}(V)}(g)=\Big\{\langle xe_1+ze_2+wf_2\rangle |
x,z,w\in \mathbb{F}, (x, z, w)\neq (0,0,0)\Big\},
\]
and its size is equal to $q^2+q+1$. This is the only case in which
the number of $1$-spaces fixed setwise by $g$ is more that
$2q+2$, and here $g$ is conjugate to a transvection (see \cite[Chapter 4]{Taylor}):
\[
\left(
\begin{array}{cccc}
1&0&0&0\\
d&1&0&0\\
0&0&1&0\\
0&0&0&1\\
\end{array}
\right)
\]
\end{itemize}
\end{proof}

\subsection{An outer automorphism of $\Sp(4, q)$ where $q=2^n$}
The following lemma is taken from
\cite{Taylor}.
\begin{lemma}\label{t11-9-new}
The following conditions hold:
\begin{itemize}
\item[$(i)$] $O(2m+1, q)$ is isomorphic to $\Sp(2m, q)$ {\rm (see
\label{T11.9}\cite[Theorem 11.9]{Taylor})}.

\item[$(ii)$] Every transvection in $O(2m+1, q)$ corresponds to a transvection in
$\Sp(2m, q)$ {\rm  (see \cite[p. 144]{Taylor}}.

\item[$(iii)$] $\PSp(4, q)\simeq P\Omega(5, q)$, for all prime powers $q$ {\rm  (see \label{C12.32}\cite[Corollary 12.32]{Taylor})}.
\end{itemize}
\end{lemma}

As pointed out by Todd \cite{todd}, using Lemma \ref{T11.9} one can
obtain some isomorphisms between $\Sp(4, 2^n)$ and $O(5, 2^n)$.
However the geometric reasons for these isomorphisms are quite
different. Taylor \cite[p. 201]{Taylor} observes that the exterior
square $\Lambda^2V$ of $V=\mathbb{F}^4$ has dimension $4$ and admits a
nondegenerate alternating form which is invariant under
$\Sp(4,q)$. Moreover each element $g$ of $\Sp(4,q)$ induces a
linear map of $\Lambda^2V$, which we denote by $\Lambda^2g$,
giving a second representation of $\Sp(4,q)$ in dimension $4$.

Let $t$ be a transvection in $\Sp(4,q)$. Then Taylor shows
\cite[p. 202]{Taylor} that the map $1-\Lambda^2t$ of $\Lambda^2V$
has image of dimension $2$, and hence also
$$
{\rm Fix}_{\Lambda^2 V}(\Lambda^2t) = {\rm Ker}_{\Lambda^2 V} (1-\Lambda^2 t),
$$
is a
$2$-dimensional subspace of $\Lambda^2V$. Since $\Lambda^2t$ is
unipotent, each $1$-space fixed setwise by $\Lambda^2t$ must be fixed pointwise,
and hence must lie in ${\rm Fix}_{\Lambda^2 V}(\Lambda^2t)$. Thus $\Lambda^2t$
fixes setwise exactly $q+1$ of the 1-spaces in $\Lambda^2V$, and hence it does not
act as a transvection on $\Lambda^2V$.

Taylor composes the representation $\Lambda^2$ with a linear map
$\Lambda^2V \rightarrow V$ to obtain an outer automorphism $\tau$
of $\Sp(4,q)$ under which $t$ is mapped to an element $t \tau$ of
$\Sp(4,q)$ with exactly $q+1$ fixed 1-spaces in $V$. Thus we have
the following result.

\begin{lemma}\label{l:trans} There exists an outer automorphism $\tau$ of $\Sp(4,q)$
under which the image of a transvection is an element with $q+1$
fixed $1$-spaces.
\end{lemma}


This is the information we need to analyse the twisted permutation code $C(T,\mathcal{I})$
for $T=\Sp(4, q)$ relative to $\mathcal{I}=(\iota, \tau)$,
with $\iota$ the identity automorphism of $T$ and $\tau$ the outer automorphism from Lemma $\ref{l:trans}$.
To construct this code, we let $V=\mathbb{F}^4$, where $\mathbb{F}$ is
a field of order $q$, we set  $m:=(q^4-1)/(q-1)$, and we choose an ordering
$(\langle v_1\rangle, \langle v_2\rangle, \ldots , \langle v_m\rangle)$
for the set ${\rm PG}(V)$  of $1$-dimensional
subspaces in $V$.  Each $g\in T$ then corresponds to the vertex
$$
\alpha(g)=\left(\langle v_1\rangle^g, \langle v_2\rangle^g,
\ldots , \langle v_m\rangle^g\right)
$$
of the Hamming graph of length $m$ over ${\rm
PG}(V)$.
If we identify $T$ with the permutation group it induces on ${\rm PG}(V)$
then $\iota$ and $\tau$ can be interpreted as permutation representations of $T$
on ${\rm PG}(V)$.
We define the permutation code as
$$
C(T)=\{\alpha(g) | g\in T\}
$$
and the twisted permutation code $C(T,\mathcal{I})$
relative to $\mathcal{I}=(\iota, \tau)$ as the code in the
Hamming graph of length $2m = 2q^3+2q^2+2q+2$ over ${\rm
PG}(V)$ consisting of the vertices
$$
\alpha(g,\mathcal{I})=(\alpha(g\iota),\alpha(g\tau))\quad \mbox{for $g\in T$.}
$$

\begin{proposition}\label{symplectic}
Let $T=\Sp(4, q)$ and $\mathcal{I}=(\iota, \tau)$ be as above.
Then the twisted permutation
code $C(T,\mathcal{I})$ has
minimum distance $2q^3+q^2$, and the difference between this
minimum distance and the lower bound given by Proposition
$\ref{lower bound}$ is $q^2$.
\end{proposition}

\begin{proof}
By Proposition \ref{lower bound}, the minimum distance
$\delta(C(T,\mathcal{I}))$ is equal to the minimum of $|\supp(g)| + |\supp(\tau(g))|$,
for $g\neq 1$ in $T$, where $\supp (g)$ denotes the subset of ${\rm PG}(V)$
consisting of the $1$-spaces which are moved by $g$ (not fixed setwise).
It follows from
Lemmas \ref{transvection} and \ref{l:trans} that if one of $g$ or $\tau(g)$
is a transvection, then $|\supp (g)|+|\supp
(\tau(g))|=(m-q^2-q-1)+(m-q-1)=2q^3+q^2$.
For all other elements $g\neq 1$, it follows from Lemma \ref{transvection}
that $|\supp(g)|+|\supp(\tau(g))|\geq 2(m-2q-2)$, which is strictly greater than
$2q^3+q^2$ for $q\geq 4$. Therefore,
$\delta(C(T,\mathcal{I}))=2q^3+q^2.$

Finally, we have from Proposition \ref{lower bound} and Lemma
\ref{transvection} that the lower bound in Proposition \ref{lower
bound} is the minimum of $2|\supp (g)|$ for nontrivial $g\in T$, and
that this is attained when $g$ is a transvection and is $2(m-q^2-q-1)=2q^3$.
The difference between $\delta(T,\mathcal{I})$ and this lower
bound is $q^2$.
\end{proof}

\subsection{Proof of Theorem \ref{main thm}.}
By considering $\delta_{tw} := \delta(C(T,\mathcal{I}))$ and
$\delta_{\rep} :=\min_{\tau\in T}\{\delta(\Rep_r(C(G,\tau))\}$,
Theorem\ref{main thm} follows from Propositions \ref{T affine} and
\ref{symplectic}.

\begin{center}
{\bf Acknowledgements}
\end{center}
The first author would like to thank the School of Mathematics and
Statistics of the University of Western Australia for their
hospitality during the preparation of this paper. She also would
like to express her deep gratitude to the second and third authors for numerous mathematical discussions.

The second author would like to acknowledge the support of a grant from the University of Western Australia associated with
the Australian Research Council Federation Fellowship FF0776186 of the third author.

\end{document}